\documentclass[10pt,bezier]{article}
\usepackage{amsmath,amssymb,amsfonts,euscript,graphicx}

\newtheorem{prethm}{{\bf Theorem}}

\newenvironment{thm}{\begin{prethm}{\hspace{-0.5
               em}{\bf.}}}{\end{prethm}}

\newtheorem{prepro}[prethm]{{\bf Theorem}}

\newtheorem{preprop}[prethm]{{\bf Proposition}}

\newtheorem{precor}[prethm]{{\bf Corollary}}

\newenvironment{cor}{\begin{precor}{\hspace{-0.5
               em}{\bf.}}}{\end{precor}}

\newtheorem{preconj}[prethm]{{\bf Conjecture}}

\newtheorem{preremark}[prethm]{{\bf Remark}}

\newenvironment{remark}{\begin{preremark}\rm{\hspace{-0.5
               em}{\bf.}}}{\end{preremark}}

\newtheorem{preexample}[prethm]{{\bf Example}}

\newtheorem{predefin}[prethm]{{\bf Definition}}

\newenvironment{defin}{\begin{predefin}\rm{\hspace{-0.5
               em}{\bf.}}}{\end{predefin}}

\newtheorem{preques}[prethm]{{\bf Question}}

\newtheorem{prelem}[prethm]{{\bf Lemma}}

\newenvironment{lem}{\begin{prelem}{\hspace{-0.5
               em}{\bf.}}}{\end{prelem}}

\newtheorem{prelam}{{\bf Lemma}}

\newtheorem{preproof}{{\bf Proof.}}

\newenvironment{proof}[1]{\begin{preproof}{\rm
               #1}\hfill{$\Box$}}{\end{preproof}}

\title{\bf \large  Application of some combinatorial arrays in\\ coloring of total graph of a commutative ring
\thanks {2010{ \it Mathematics Subject Classification}: 05B15, 05C15, 05C25, 05C69, 16N40.}}

\author{{\normalsize { G. Aalipour${}^{\mathsf{a,b}}$} and { S. Akbari${}^{\mathsf{a,b}}$}}\vspace{3mm}\\
{\footnotesize{${}^{\mathsf{a}}$ Department of Mathematical
 Sciences, Sharif University of Technology, Tehran, Iran}}\\
{\footnotesize{${}^{\mathsf{b}}$ School of Mathematics, Institute
for Research in Fundamental Sciences, (IPM),}}\\
{\footnotesize{}P.O. Box 19395-5746, Tehran, Iran}\\\\
{\footnotesize{$\mathsf{alipour\_ghodrat@mehr.sharif.ir}$\quad\quad
$\mathsf{s\_akbari@sharif.edu}$\quad\quad}}}
\date{}
\begin{document}
\maketitle
\vspace{-1cm}
\begin{abstract}
{\small Let  $R$  be a commutative ring with unity and $Z(R)$ and ${\rm Reg}(R)$ be the set of zero-divisors and non-zero zero-divisors of $R$, respectively.
We denote by $T(\Gamma(R))$, the total graph of $R$, a simple graph with the vertex set $R$ and two distinct vertices $x$ and $y$ are adjacent if and only if $x+y\in Z(R)$. The induced subgraphs on $Z(R)$ and  ${\rm Reg}(R)$ are denoted by $Z(\Gamma(R))$ and $Reg(\Gamma(R))$, respectively. These graphs were first introduced by D.F. Anderson and A. Badawi in $2008$. In this paper, we prove the following result: let $R$ be a finite ring and one of the following conditions hold: (i) The residue field of $R$ of minimum size has even characteristic, (ii) Every residue field of $R$ has odd characteristic and $\frac{R}{J(R)}$ has no summand isomorphic to $\mathbb{Z}_3\times \mathbb{Z}_3$,
then the chromatic number and clique number of $T(\Gamma(R))$ are equal to $\max\{|\mathfrak{m}|\,:\, \mathfrak{m}\in {\rm Max}(R)\}$. The same result holds for $Z(\Gamma(R))$. Moreover, if the residue field of $R$ of minimum size has even characteristic or every residue field of $R$ has odd characteristic, then we determine
the chromatic number and clique number of $Reg(\Gamma(R))$ as well.}
\end{abstract}

\vspace{4mm} \noindent{\bf\large 1. Introduction}\vspace{4mm}\\
{Throughout this paper, all rings are assumed to be
commutative with unity. Let $R$ be a ring. We denote by $U(R)$, $Z(R)$, $Z^*(R)$, ${\rm Reg}(R)$, ${\rm{Min}}(R)$, ${\rm{Spec}}(R)$ and ${\rm{Max}}(R)$,
the set of invertible elements, zero-divisors, non-zero zero-divisors, regular elements, minimal prime ideals,
prime ideals and maximal ideals of $R$,
respectively. The {\it{Jacobson radical}} and the {\it{nilradical}}
of $R$ are denoted by $J(R)$ and $Nil(R)$, respectively. The ring
$R$ is said to be \textit{reduced} if it has no non-zero nilpotent
element.  The \textit{Krull dimension} of $R$ is denoted by $dim(R)$.
A {\it{local ring}} is a ring with exactly one maximal ideal. A ring with finitely many maximal ideals is called a {\it{semi-local ring}}. The set of \textit{associated prime ideals of an $R$-module $R$} is denoted by ${\rm Ass}(R)=\{\mathfrak{p}\in {\rm{Spec}}(R)\, :\, \mathfrak{p}=Ann(x),\ \text{for some}\ x\in R\}$. For classical theorems and notations in commutative algebra, the interested reader is referred to \cite{ati} and \cite{kaplansky}.

Let $G$ be a graph with the vertex set $V(G)$. The \textit{complement} of $G$ is denoted by $\overline{G}$. If $G$ is connected, then we mean by  $diam(G)$, the {\textit{diameter}} of $G$. If $G$ is not connected, then $diam(G)$ is defined to be $\infty$. We denote by $K_X$ and $K_{X,Y}$, the \textit{complete graph} with the vertex set $X$ and the \textit{complete bipartite graph} with two parts $X$ and $Y$, respectively. The \textit{direct product} (sometimes called \textit{Kronecker product} or \textit{tensor product}) of two graphs $G$ and $H$, denoted by $G\times H$, is a graph with the vertex set $V(G)\times V(H)$ and two distinct vertices $(x_1,y_1)$ and $(x_2,y_2)$ are adjacent if and only if $x_1$ and $x_2$ are adjacent in $G$ and $y_1$ and $y_2$ are adjacent in $H$. A \textit{clique} in a graph $G$ is a subset of pairwise adjacent vertices and the supremum of the size of cliques in $G$, denoted by
$\omega(G)$, is called the \textit{clique number} of $G$. By
$\chi(G)$, we denote \textit{the chromatic number} of $G$, i.e. the
minimum number of colors which can be assigned to the vertices of
$G$ in such a way that every two adjacent vertices have different
colors.
A coloring of the vertices that any two adjacent vertices have different colors, is called a \textit{proper vertex coloring}.
Let $r\leq n$ be two positive integers. A \textit{Latin rectangle} is an $r\times n$ matrix whose entries are $n$ distinct symbol with no symbol occurring more than once in any row or column. 

For a commutative ring $R$, the \textit{total graph} of $R$, denoted by $T(\Gamma(R))$, is a graph with the vertex set $R$ and two distinct vertices $x$ and $y$ are adjacent if and only if $x+y\in Z(R)$. This is the \textit{Cayley sum graph} (or sometimes is called \textit{addition Cayley graph}) $Cay^{+}(R,Z(R))$.
For more information on Cayley sum graphs see \cite{cayley-sum1}, \cite{cayley-sum2} and the other references there.
The authors in \cite{ak} and \cite{totalgraph} have studied $T(\Gamma(R))$ and two its subgraphs $Reg(\Gamma(R))$ and $Z(\Gamma(R))$, the induced subgraphs of $T(\Gamma(R))$ on $Reg(R)$ and $Z(R)$, respectively. Some other properties of $Reg(\Gamma(R))$ are studied in \cite{akbari-heydari}. The authors in \cite{akbari-heydari}, characterize all rings $R$ such that $2\notin Z(R)$ and $Reg(\Gamma(R))$ is a complete graph. They also prove that if $Reg(\Gamma(R))$ is a tree, then it has at most two vertices. In this paper, we investigate the clique number and the chromatic number of $T(\Gamma(R))$ and its two subgraphs $Z(\Gamma(R))$ and $Reg(\Gamma(R))$ as well.

\vspace{4mm} \noindent{\bf\large 2. Preliminaries on the Total Graph of a Ring}\vspace{4mm}\\
In this section, for a commutative ring $R$, we provide some preliminary lemmas on $T(\Gamma(R))$.

\begin{lem}\label{zerodivisorsofzero-dimensionals}{\rm (\cite[Theorem 91]{kaplansky})}
Let $R$ be a zero-dimensional ring. Then $Z(R)$ is the union of all maximal ideals of $R$. In particular, $Reg(R)=U(R)$.
\end{lem}
\begin{lem}\label{shiftingbynilpotents}
Let $R$ be a zero-dimensional ring, $x\in R$ and $a\in Nil(R)$. Then $x+a\in Z(R)$ if and only if $x\in Z(R)$.
\end{lem}
\begin{proof}
{Let $u$ be an arbitrary element of $R$. It is clear that $u\in U(R)$ if and only if $u+a\in U(R)$. Since $Z(R)=R\setminus Reg(R)$,  the assertion immediately follows from Lemma \ref{zerodivisorsofzero-dimensionals}.}
\end{proof}
\begin{lem}\label{zerodivisorsofreducedsareshifted}
Let $R$ be a zero-dimensional ring and $x\in R$. Then $x+Nil(R)\in Z(\frac{R}{Nil(R)})$ if and only if $x\in Z(R)$.
\end{lem}
\begin{proof}
{Since $R$ is zero-dimensional, we have $Nil(R)=J(R)$ and ${\rm Max}(\frac{R}{Nil(R)})=\{\,\frac{\mathfrak{m}}{Nil(R)}\,|\, \mathfrak{m}\in {\rm Max}(R)\}$. Thus, for a maximal ideal $\mathfrak{m}$ of $R$, we have $x\in \mathfrak{m}$ if and only if $x+Nil(R)\in \frac{\mathfrak{m}}{Nil(R)}$. 
Since $\frac{R}{Nil(R)}$ is zero-dimensional, by Lemma \ref{zerodivisorsofzero-dimensionals}, we conclude that $x+Nil(R)\in Z(\frac{R}{Nil(R)})$ if and only if $x\in Z(R)$. The proof is complete.}
\end{proof}
\begin{lem}\label{zerodivisorsareideal}{\rm (\cite[Theorems 2.1 and 2.2]{totalgraph})}
Let $R$ be a ring and $Z(R)$ be an ideal of $R$. Then the following statements hold.

\noindent {\rm{(i)}} If $2\in Z(R)$, then $T(\Gamma(R))$ is a disjoint union of $|\frac{R}{Z(R)}|$ complete graphs $K_{|Z(R)|}$ on a coset of $Z(R)$.\\
{\rm{(ii)}} If $2\notin Z(R)$, then $T(\Gamma(R))$ is a disjoint union of $K_{|Z(R)|}$ and $\frac{|{R}/{Z(R)}|-1}{2}$ complete bipartite graphs $K_{|Z(R)|,|Z(R)|}$.
\end{lem}

\vspace{4mm} \noindent{\bf\large 3. Coloring of the total graph of a ring}\vspace{4mm}\\
In this section we study the chromatic number of the total graph of a finite commutative ring.  We conjecture that for every finite ring $R$, the following equalities hold:
\begin{equation}\label{identity-conjecture}
\chi(T(\Gamma(R)))=\omega(T(\Gamma(R)))=
  \begin{cases}
     4\,;      & R\cong \mathbb{Z}_3\times \mathbb{Z}_3,\,\, \\
    \max\{|\mathfrak{m}|\,:\, \mathfrak{m}\in {\rm Max}(R)\};       &{\rm otherwise}.\
\end{cases}
\end{equation}

\begin{remark}
In order to determine the clique number and the chromatic number of $T(\Gamma(R))$, we assume that $R$ is a finite ring because if $R$ is an infinite ring which is not a domain, then for every non-zero zero-divisor $x$ of $R$, $\frac{R}{Ann(x)}\cong Rx$, as $R$-modules. Since both $Rx$ and $Ann(x)$ form a clique for $Z(\Gamma(R))$, we deduce that $Z(\Gamma(R))$ has an infinite clique and so $T(\Gamma(R))$ does.
\end{remark}
In the sequel, we prove (\ref{identity-conjecture}) for many families of rings. In this direction, first we reduce the problem of coloring of the total graph of a finite
ring to the problem of coloring of total graph of a finite reduced ring.

\begin{defin}
Let $G$ be a graph, $V(G)=\{v_1,\ldots,v_n\}$, $H_{1},\ldots,H_{n}$ be $n$ graphs of order $m$. Let $G(H_{1},\ldots,H_{n})$ be a graph with the vertex set $\cup_{i=1}^nV(H_{i})$ such that:

\noindent (i) Any vertex of $H_{i}$ is adjacent to any vertex of $H_{j}$ if and only if $v_i$ is adjacent to $v_j$ in $G$.\\
(ii) Two vertices of $H_{i}$ are adjacent in $G(H_{1},\ldots,H_{n})$ if and only if they are adjacent in $H_{i}$.

If for every $i$, $1\leq i\leq n$, $H_{i}$ has no edge, then the resulting graph is called a \textit{balanced blow-up}. If for every $i$, $H_{i}\cong K_m$, then we denote $G(H_{1},\ldots,H_{n})$ by $\widehat{G}(m)$. We state the following simple lemma without proof.
\end{defin}

\begin{lem}\label{like-blow-up}
Let $G$ be a graph and $H_1,\ldots,H_n$ be simple graphs of order $m$. Then
$$\chi\big(G(H_{1},\ldots,H_{n})\big)\leq m\chi(G).$$
\end{lem}
\begin{lem}\label{coloring-reduction-to-reduced-rings}
If {\rm(\ref{identity-conjecture})} holds for every finite reduced ring, then it holds for every finite ring.
\end{lem}

\begin{proof}
{Let $R$ be a finite ring, $|\frac{R}{J(R)}|=k$ and $\frac{R}{J(R)}=\{f_1+J(R),\ldots,f_k+J(R)\}$. For $1\leq i\leq k$, we define $H_{i}$ as $K_{|J(R)|}$ if $2f_i\in Z(R)$ and as $\overline{K_{|J(R)|}}$, otherwise. By Lemma \ref{zerodivisorsofreducedsareshifted},
$$T(\Gamma(R))\cong T\big(\Gamma(\frac{R}{J(R)})\big)\big(H_{1},\ldots,H_{k}\big).$$
Thus, by Lemma \ref{like-blow-up},
\begin{equation}\label{blowup-total-graph}
\chi\Big(T\big(\Gamma(R)\big)\Big)\leq |J(R)|\chi\Big(T\big(\Gamma(\frac{R}{J(R)})\big)\Big).
\end{equation}

Now, we are ready to prove the assertion. Suppose that (\ref{identity-conjecture}) holds for every finite reduced ring and let $R$ be a finite ring such that $\frac{R}{J(R)}$ is not isomorphic to $\mathbb{Z}_3\times \mathbb{Z}_3$. Since $$\chi(T(\Gamma(\frac{R}{J(R)})))=\max\{|\mathfrak{m}|\,:\, \mathfrak{m}\in {\rm Max}(\frac{R}{J(R)})\},$$
by (\ref{blowup-total-graph}), we obtain that
$$\chi(T(\Gamma(R)))\leq \max\{|\mathfrak{m}|\,:\, \mathfrak{m}\in {\rm Max}(R)\}.$$
On the other hand, by Lemma \ref{zerodivisorsofzero-dimensionals}, $\omega(T(\Gamma(R)))\geq\max\{|\mathfrak{m}|\,:\, \mathfrak{m}\in {\rm Max}(R)\}$. Hence
$$\chi(T(\Gamma(R)))=\omega(T(\Gamma(R)))=\max\{|\mathfrak{m}|\,:\, \mathfrak{m}\in {\rm Max}(R)\}.$$
Now, assume that $\frac{R}{J(R)}\cong \mathbb{Z}_3\times \mathbb{Z}_3$. If $J(R)=\{0\}$, then using Figure $1$, $\chi(T(\Gamma(R)))=\omega(T(\Gamma(R)))=4$.
\\
\centerline{\includegraphics[scale=.35, angle=0, trim = .5cm 11cm 1cm 4.5cm, clip]{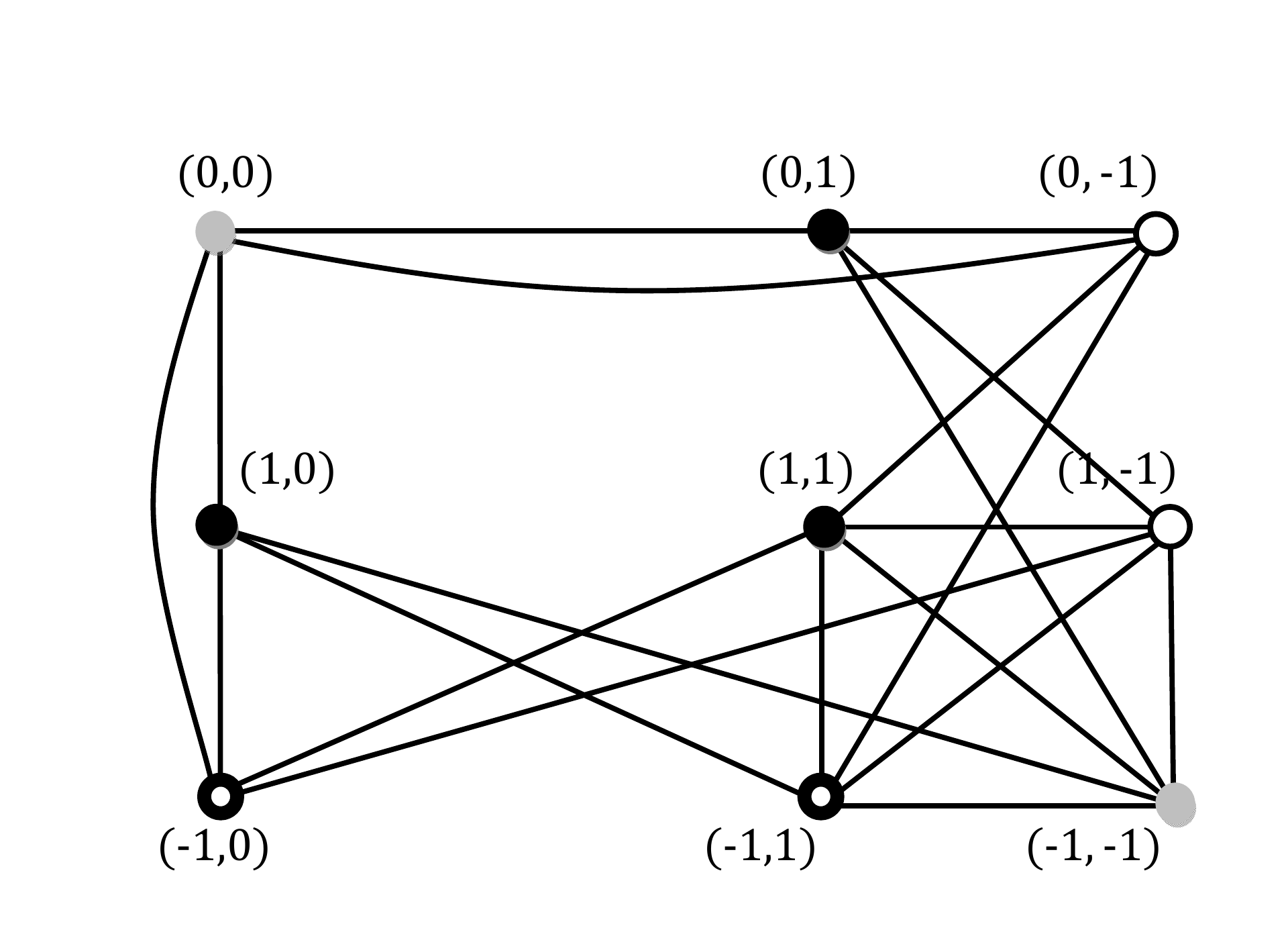}}
\centerline{$\text{\bf Figure 1. }$ A coloring of $T(\Gamma(\mathbb{Z}_3\times \mathbb{Z}_3))$ using $4$ colors.}
\\

Thus, assume that $|J(R)|\geq 2$. Since by Chinese Reminder Theorem (see \cite[Proposition 1.10]{ati}), $\frac{R}{J(R)}\cong \frac{R}{\mathfrak{m}_1}\times \frac{R}{\mathfrak{m}_2}$, where ${\rm Max(R)}=\{\mathfrak{m}_1,\mathfrak{m}_2\}$, we may assume that $R$ is partitioned into
$(0,0)+J(R),(0,1)+J(R),(0,-1)+J(R),(1,0)+J(R),(1,1)+J(R),(1,-1)+J(R),(-1,0)+J(R),(-1,1)+J(R),(-1,-1)+J(R)$. Now, we color the the complete graph induced on $(0,0)+J(R)$ by colors $a_1,\ldots, a_{|J(R)|}$, the complete graphs induced on $(0,1)+J(R)$ and $(1,0)+J(R)$ by $a'_1,\ldots, a'_{|J(R)|}$ and the complete graphs induced on $(0,-1)+J(R)$ and $(-1,0)+J(R)$ by $a''_1,\ldots, a''_{|J(R)|}$. Now, the remaining vertices form a complete $4$-partite graph with parts $(1,1)+J(R),(1,-1)+J(R),(-1,1)+J(R),(-1,-1)+J(R)$. We color the vertices in part $(1,1)+J(R)$ by $a'_1$, the vertices in part $(-1,-1)+J(R)$ by $a''_1$, the vertices in part $(1,-1)+J(R)$ by $a_1$ and the vertices in part $(-1,1)+J(R)$ by $a_2$. Using Lemma \ref{zerodivisorsofreducedsareshifted}, it can be seen that this is a proper coloring of $T(\Gamma(R))$ with $3|J(R)|=\max\{|\mathfrak{m}|\,:\, \mathfrak{m}\in {\rm Max}(R)\}$ colors. The proof is complete.}
\end{proof}

\begin{defin}
Suppose that $F_1$ and $F_2$ are two finite fields, $|F_1|\leq|F_2|$, $A\subseteq F_1$, $B\subseteq F_2$  and $|A|\leq |B|$.  Let $L$ be a table of size $|A|\times |B|$ over alphabet $C$ whose rows are labelled by the elements of $A$ and columns are labelled by the elements of $B$. The table $L$ is called a ($|A|\times |B|$) \textit{Latin-sum array} over $C$ if for every two distinct pairs  $(x_1,y_1)$ and $(x_2,y_2)$ of elements of $A\times B$ both following conditions hold:

\noindent
(i) If $x_1+x_2=0$ in $F_1$, then $L_{x_1,y_1}\neq L_{x_2,y_2}$.  \\
(ii) If $y_1+y_2=0$ in $F_2$, then $L_{x_1,y_1}\neq L_{x_2,y_2}$.
\end{defin}

\begin{remark}
If both fields have characteristic $2$, then every Latin rectangle has this property. The problem is the existence of these arrays for the other characteristics.
\end{remark}

\begin{lem}\label{Latin-sum array}
For every two finite fields $F_1$ and $F_2$ with $|F_1|\leq|F_2|$, there exists a $|F_1|\times |F_2|$ Latin-sum array over alphabet $C$, where
$$|C|=  \begin{cases}
     4\,;      & F_1=F_2=\mathbb{Z}_3,\,\, \\
    |F_2|       &{\rm otherwise}.\
\end{cases}$$
\end{lem}

\begin{proof}
{We consider three following cases.

{\bf Case 1.} First suppose that both $F_1$ and $F_2$ are of characteristic $2$. Thus, as we mentioned above, every Latin rectangle over $F_2$ has this property.

{\bf Case 2.} Now, suppose that both have odd characteristic. Let $n=\frac{|F_1|-1}{2}$ and $m=\frac{|F_2|-1}{2}$. Let $F_1=\{0,x_1,-x_1,\ldots,x_n,-x_n\}$ and $F_2=\{0,y_1,-y_1,\ldots,y_m,-y_m\}$. Now, we construct a $|F_1|\times |F_2|$ table over $\{0,\pm 1,\ldots,\pm m\}$, say $L$, whose rows are labelled by $0,x_1,-x_1,\ldots,x_n,-x_n$ and columns are labelled by $0,y_1,-y_1,\ldots,y_m,-y_m$, respectively. To construct $L$ as a Latin-sum array over a set $C$, we should arrange the elements of $C$ in such a way that the following conditions are satisfied:
\begin{description}
  \item[C1:] The entries of the first row are pairwise distinct,
  \item[C2:] The entries of the first column are pairwise distinct,
  \item[C3:]  The rows corresponding to $x_i$ and $-x_i$ have no common entry, for $1\leq i\leq n$,
  \item[C4:] The columns corresponding to $y_j$ and $-y_j$ have no common entry, for $1\leq j\leq m$.
\end{description}

\noindent
To satisfy these four conditions, it suffices to consider one of the following $|F_1|\times |F_2|$ arrays. Note that Table 1 involves the case $F_1=F_2=\mathbb{Z}_3$. In this case, the alphabet $C$ has $4$ elements but for $|F_2|>3$, it has exactly $|F_2|$ elements. 


\begin{center}
\begin{tabular}{|c||cc|cc|c|cc|}
\hline
   $0$&$1$&$-1$&$2$&$-2$&\ldots&$m$&$-m$ \\
   \hline\hline
   $1$&$1$&$-1$&$1$&$-1$&\ldots&$1$&$-1$ \\
   $2$ & $2$ & $0$ & $2$ & $0$  & \ldots & $2$ & $0$ \\
    \hline\hline
   $-1$ & $1$ & $-1$ & $1$ & $-1$  & \ldots & $1$ & $-1$ \\
   $-2$ & $2$ & $0$ & $2$ & $0$  & \ldots & $2$ & $0$ \\
  \hline\hline
 \vdots & \vdots &\vdots & \vdots &\vdots & $\ddots$ &\vdots &\vdots \\
  \hline\hline
  $-(n-2)$ & $1$ & $-1$ & $1$ & $-1$  & \ldots & $1$ & $-1$ \\
  $-(n-1)$ & $2$ & $0$ & $2$ & $0$  & \ldots & $2$ & $0$ \\
  \hline\hline
  $n$ & $1$ & $-1$ & $1$ & $-1$  & \ldots & $1$ & $-1$ \\
  $-n$ & $2$ & $0$ & $2$ & $0$  & \ldots & $2$ & $0$ \\
   \hline
\end{tabular}
\end{center}
\centerline{$\text{\bf Table 1. }$ A  $|F_1|\times |F_2|$ Latin-sum array in the case $n$ is odd.}

\begin{center}
\begin{tabular}{|c||c c|c c|c|c c|}
\hline
   $0$ & $1$ & $-1$ & $2$ & $-2$  & \ldots & $m$ & $-m$ \\
   \hline\hline
   $1$ & $1$ & $-1$ & $1$ & $-1$  & \ldots & $1$ & $-1$ \\
   $2$ & $2$ & $-2$ & $2$ & $-2$  & \ldots & $2$ & $-2$ \\
    \hline\hline
   $-1$ & $1$ & $-1$ & $1$ & $-1$  & \ldots & $1$ & $-1$ \\
   $-2$ & $2$ & $-2$ & $2$ & $-2$  & \ldots & $2$ & $-2$ \\
  \hline\hline
 \vdots & \vdots &\vdots & \vdots &\vdots & $\ddots$ &\vdots &\vdots \\
 \hline\hline
  $n-1$ & $1$ & $-1$ & $1$ & $-1$  & \ldots & $1$ & $-1$ \\
  $n$ & $2$ & $-2$ & $2$ & $-2$  & \ldots & $2$ & $-2$ \\
  \hline\hline
  $-(n-1)$ & $1$ & $-1$ & $1$ & $-1$  & \ldots & $1$ & $-1$ \\
  $-n$ & $2$ & $-2$ & $2$ & $-2$  & \ldots & $2$ & $-2$ \\
   \hline
\end{tabular}
\end{center}
\centerline{$\text{\bf Table 2. }$ A  $|F_1|\times |F_2|$ Latin-sum array in the case $n$ is even.}
\vspace{4mm}
{\bf Case 3.} Now, assume that exactly one $F_i$, for $i=1,2$, is of characteristic $2$. First suppose that $char(F_1)=2$ and $F_2$ is of odd characteristic. Let $|F_1|=n$, $F_1=\{x_1,x_2,\ldots,x_{n}\}$ and $F_2=\{0,y_1,-y_1,\ldots,y_m,-y_m\}$. To continue the proof, we prove the following stronger assertion: we construct a $|F_2|\times |F_2|$ table $L$ whose rows are labeled by $x_1,\ldots,x_n,\ldots,x_{2m+1}$ and columns are labeled by $0,y_1,-y_1,\ldots,y_m,-y_m$. We arrange the entries of $L$ in such a way that the following conditions are satisfied:
\begin{description}
  \item[D1:] The entries of each row are pairwise distinct,
  \item[D2:] The entries of the first column are pairwise distinct,
  \item[D3:] The columns corresponding to $y_j$ and $-y_j$ have no common entry, for $1\leq j\leq m$.
\end{description}
Regarding {\bf D1}, we define $1,2,\ldots,2m+1$ as the first row. We construct the $i$-th row from the ($i-1$)-th row by changing two symbols in the ($i-1$)-th row such that the conditions {\bf D1}, {\bf D2} and {\bf D3} are satisfied for the $j$-th row, $j=1,\ldots,i$. To construct the second row, we replace the symbol $1$ with $2$ in the first row and consider the new $1\times (2m+1)$ vector as the second row. To construct the third row, we replace the first entry of the second row by $3$ and then change two symbols $2$ and $4$ with each other (as $L_{1,2}=2$, to satisfy {\bf D3}). It can be easily checked that the entries in each row are pairwise distinct and {\bf D3} is satisfied in each step. We construct the fourth row from the third row by replacing the first entry of the third row by $4$. For the fifth row we do similarly except that we make two changes. We can continue this procedure to construct the $2m$-th row, i.e. in the Step $2i$ we make just two changes and in the Step $2i+1$ we make three changes, for $1\leq i\leq m$.  For the ($2m+1$)-th row, we first replace the first entry of the $2m$-th row with the $2m+1$. Now, $2m$ appears in the two last columns. Now, we change $2m$ and $1$. We obtain the desired Latin-sum array. In the below, we can see the resulting Latin-sum array for $|F_2|=7$.

$$\begin{tabular}{|c|| c c |c c | c c |}
   \hline
   $1$ & $2$ & $3$ & $4$ & $5$  & $6$ & $7$ \\

   $2$ & $1$ & $3$ & $4$ & $5$  & $6$ & $7$ \\
   $3$ & $1$ & $4$ & $2$ & $5$  & $6$ & $7$ \\

   $4$ & $1$ & $3$ & $2$ & $5$  & $6$ & $7$ \\
   $5$ & $1$ & $3$ & $2$ & $6$  & $4$ & $7$ \\

   $6$ & $1$ & $3$ & $2$ & $5$ & $4$ & $7$ \\
   $7$ & $6$ & $3$ & $2$ & $5$ & $4$ & $1$  \\
   \hline
\end{tabular}$$

Now, suppose that $F_1$ has odd characteristic and $F_2$ has characteristic $2$. Let $L$ be the above $|F_2|\times |F_2|$ array in which conditions {\bf D1}, {\bf D2} and {\bf D3} hold. Thus, the first $|F_1|$ rows of the transpose of $L$ is the desired Latin-sum array. The proof is complete.}
\end{proof}

Now, we determine the chromatic number of the total graph of a direct product of finitely many fields.
\begin{lem}\label{chromatic-no-of-reduced-rings}
Let $n\geq 2$ be a positive integer and $F_1,\ldots,F_n$ be finite fields such that $|F_1|\leq\cdots\leq |F_n|$. Suppose that one of the following conditions holds:

\noindent {\rm (i)} The field $F_1$ has even characteristic,\\
{\rm (ii)} Every $F_i$ has odd characteristic and $F_1\times F_2\ncong\mathbb{Z}_3\times \mathbb{Z}_3$.

\noindent Then $\chi\big(T(\Gamma(F_1\times\cdots\times F_n))\big)=|F_2|\cdots|F_n|$.
\end{lem}

\begin{proof}
{For $i=2,\ldots,n$, let $L^{F_1,F_i}$ be the Latin-sum array deduced from Lemma \ref{Latin-sum array}. We define the following map $f$ on $F_1\times\cdots\times F_n$ as follows: $$f\big((x_1,\ldots,x_n)\big)=\big(L^{F_1,F_2}_{x_1,x_2},\ldots,L^{F_1,F_{n}}_{x_1,x_{n}}\big),$$
where $L^{F_1,F_k}_{x_i,x_j}$ is the  $(x_i,x_j)$-entry of $L^{F_1,F_k}$. Now, we prove that $f$ is a proper vertex coloring for $T(\Gamma(F_1\times\cdots\times F_n))$ in both Cases (i) and (ii). First, suppose that Case (i) holds. Let $(x_1,\ldots,x_n)$ and $(y_1,\ldots,y_n)$ be two adjacent vertices of $T(\Gamma(F_1\times\cdots\times F_n))$ with the same color. Hence, there exists $1\leq j\leq n$, such that $x_j+y_j=0$. First assume that $j=1$. Since $L^{F_1,F_i}$ is a Latin-sum array and $f\big((x_1,\ldots,x_n)\big)=f\big((y_1,\ldots,y_n)\big)$, we conclude that $x_i=y_i$, for $i=2,\dots,n$. Thus, $(x_1,\ldots,x_n)=(y_1,\ldots,y_n)$. Now, assume that $j\geq 2$. If $x_j=y_j$, then by $L^{F_1,F_j}_{x_1,x_j}=L^{F_1,F_j}_{y_1,y_j}$, we deduce that $x_1=y_1$ and so for every $i$, $x_i=y_i$. Hence suppose that $x_j=-y_j\in F_j\setminus\{0\}$. Since $L^{F_1,F_j}$ is a Latin-sum array and $L^{F_1,F_j}_{x_1,x_j}=L^{F_1,F_j}_{y_1,y_j}$, we get a contradiction. Thus, in Case (i), $f$ is a proper vertex coloring with $|F_2|\times\cdots\times |F_n|$ colors. On the other hand, $\{0\}\times F_2\times\cdots\times F_n$ is a clique for $T(\Gamma(F_1\times\cdots\times F_n))$. Thus, $\chi\big(T(\Gamma(F_1\times\cdots\times F_n))\big)=|F_2|\cdots|F_n|$. Case (ii) can be similarly proved. The proof is complete.}
\end{proof}
Now, we have the following theorem.
\begin{thm}\label{chromatic-no}
Let $R$ be a finite ring such that one of the following conditions holds:

\noindent {\rm (i)} The residue field of $R$ of minimum size has even characteristic,\\
{\rm (ii)} Every residue field of $R$ has odd characteristic and $\frac{R}{J(R)}$ has no summand isomorphic to $\mathbb{Z}_3\times \mathbb{Z}_3$.

\noindent Then the following equalities hold:
$$\chi(T(\Gamma(R)))=\omega(T(\Gamma(R)))=\chi(Z(\Gamma(R)))=\omega(Z(\Gamma(R)))=\max\{|\mathfrak{m}|\,:\, \mathfrak{m}\in {\rm Max}(R)\}.$$
\end{thm}
\begin{proof}
{By Lemma \ref{coloring-reduction-to-reduced-rings}, Chinese Reminder Theorem (see \cite[Proposition 1.10]{ati}) and Lemma \ref{chromatic-no-of-reduced-rings}, we find that  $\chi(T(\Gamma(R)))=\omega(T(\Gamma(R)))=\max\{|\mathfrak{m}|\,:\, \mathfrak{m}\in {\rm Max}(R)\}$. On the other hand, by Lemma \ref{zerodivisorsofzero-dimensionals}, $\omega(Z(\Gamma(R)))\geq\max\{|\mathfrak{m}|\,:\, \mathfrak{m}\in {\rm Max}(R)\}$. The proof is complete.}
\end{proof}

\begin{remark}
Here, as an example of the correctness of (1), we will provide a proper vertex coloring of $T(\Gamma(\mathbb{Z}_3\times \mathbb{Z}_3\times \mathbb{Z}_3))$ using $9$ colors which is not deduced from the previous theorems. The coloring classes are as follow.
$$\begin{tabular}{l l}

\{$(1,0,1), (1,1,0), (0,1,1), (1,1,1)$\} & \{$(-1,0,1), (-1,1,1), (0,1,0)$\} \\
\{$(-1,0,-1), (0,-1,0), (-1,-1,-1)$\} &  \{$(-1,0,0), (0,-1,1)$\} \\
\{$(1,0,0), (0,1,-1), (1,1,-1)$\} & \{$(-1,1,0), (0,0,-1), (-1,1,-1)$\} \\
\{$(-1,-1,0), (-1,-1,1), (0,0,1)$\} & \{$(0,0,0), (1,-1,1)$\}\\
\{$(1,-1,0), (0,-1,-1), (1,-1,-1), (1,0,-1)$\} &
\end{tabular}$$

\end{remark}
\vspace{4mm} \noindent{\bf\large 4. The induced subgraph on regular elements}\vspace{4mm}\\
In \cite{akbari-heydari}, for a Noetherian ring $R$ with $2\notin Z(R)$, the clique number and the chromatic number of $Reg(\Gamma(R))$ are studied. Indeed, the following result was proved.
\begin{thm}\label{thm-akbari-heydari}{\rm (\cite[Theorem 6]{akbari-heydari})}
Let $R$ be a ring and $2\notin Z(R)$. If $Z(R)=\cup_{i=1}^nP_i$, where $P_1,\ldots,P_n$ are prime ideals of $R$ and $Z(R)\neq\cup_{i\neq j}P_i$, for $j=1,\ldots,n$, then $\chi\Big(Reg\big(\Gamma(R)\big)\Big)=\omega\Big(Reg\big(\Gamma(R)\big)\Big)=2^n$.
\end{thm}
\begin{cor}
Let $R$ be a finite ring. If every residue field of $R$ has odd characteristic, then
$$\chi\Big(Reg\big(\Gamma(R)\big)\Big)=\omega\Big(Reg\big(\Gamma(R)\big)\Big)=2^{|{\rm Max}(R)|}.$$
\end{cor}
\begin{proof}
{Since $R$ is finite, by Lemma \ref{zerodivisorsofzero-dimensionals}, $Z(R)=\cup _{i=1}^n\mathfrak{m}_i$, where ${\rm Max}(R)=\{\mathfrak{m}_1,\ldots,\mathfrak{m}_n\}$. Since every residue field of $R$ has odd characteristic, we deduce that $2\notin \cup _{i=1}^n\mathfrak{m}_i$.  Now using the Prime Avoidance Theorem (\cite[Theorem 3.61]{sharp}), the assertion immediately follows from Theorem \ref{thm-akbari-heydari}.}
\end{proof}
Similar to the previous results, we are going to determine the clique number and the chromatic number of $Reg(\Gamma(R))$, for a finite ring $R$ with $2\in Z(R)$. We start with the following lemma.
\begin{lem}\label{coloring-reduction-to-reduced-rings for regs}
Let $R$ be a finite ring. Then the following holds.
$$\chi\Big(Reg\big(\Gamma(R)\big)\Big)\leq |J(R)|\chi\Big(Reg\big(\Gamma(\frac{R}{J(R)})\big)\Big).$$
Moreover, if $2\notin Z(R)$, then $\chi\Big(Reg\big(\Gamma(R)\big)\Big)=\chi\Big(Reg\big(\Gamma(\frac{R}{J(R)})\big)\Big)$.
\end{lem}

\begin{proof}
{The proof of the first statement is like to the proof of Lemma \ref{coloring-reduction-to-reduced-rings}. To prove the second statement, note that if $2\notin Z(R)$, then $Reg\big(\Gamma(R)\big)$ is a balanced blow-up of $Reg\big(\Gamma(\frac{R}{J(R)})\big)$. Hence, $\chi\Big(Reg\big(\Gamma(R)\big)\Big)=\chi\Big(Reg\big(\Gamma(\frac{R}{J(R)})\big)\Big)$.}
\end{proof}


\begin{lem}\label{latin-sum array for regs}
Let $F_1$ and $F_2$ be two finite fields with $|F_1|\leq |F_2|$. If $char(F_1)=2$, then there exists a $(|F_1|-1)\times(|F_2|-1)$ Latin-sum array over $F_2\setminus \{0\}$.
\end{lem}

\begin{proof}
{If both $F_1$ and $F_2$ have characteristic $2$, then it suffices to consider a $(|F_1|-1)\times(|F_2|-1)$ Latin rectangle. Otherwise, construct a $(|F_1|-1)\times(|F_2|-1)$ table whose every row is $a_1,\ldots,a_{|F_2|-1}$, where $F_2=\{0,a_1,\ldots,a_{|F_2|-1}\}$. In this table, the rows are labelled by the elements of $F_1\setminus\{0\}$ and the columns are labelled by the elements of $F_2\setminus \{0\}$.}
\end{proof}

The proof of following result is similar to the proof of Lemma \ref{chromatic-no-of-reduced-rings}.
\begin{lem}\label{chromatic-no-of-reduced-rings for regs}
Let $n\geq 2$ be a positive integer and $F_1,\ldots,F_n$ be finite fields with $|F_1|\leq\cdots\leq |F_n|$. If $char(F_1)=2$, then  $$\chi\big(Reg(\Gamma(F_1\times\cdots\times F_n))\big)=\omega\big(Reg(\Gamma(F_1\times\cdots\times F_n))\big)=(|F_2|-1)\cdots(|F_n|-1).$$
\end{lem}
\begin{proof}
{For $i=2,\ldots,n$, let $L^{F_1^*,F_i^*}$ be the $(|F_1|-1)\times(|F_i|-1)$ Latin-sum array over $F_i^*=F_i\setminus \{0\}$ deduced from Lemma \ref{latin-sum array for regs}.  We define the following map $g$ on $F_1\setminus\{0\}\times\cdots\times F_n\setminus\{0\}$ as follows: $$g\big((x_1,\ldots,x_n)\big)=\big(L^{F_1^*,F_2^*}_{x_1,x_2},\ldots,L^{F_1^*,F_n^*}_{x_1,x_{n}}\big),$$
where $L^{F_1^*,F_k^*}_{x_i,x_j}$ is the  $(x_i,x_j)$-entry of $L^{F_1^*,F_k^*}$. Similar to the proof of Lemma \ref{chromatic-no-of-reduced-rings}, one may prove that $g$ is a proper vertex coloring for $Reg(\Gamma(F_1\times\cdots\times F_n))$. On the other hand, $\{1\}\times F_2\setminus\{0\}\times\cdots\times F_n\setminus\{0\}$ is a clique for $Reg(\Gamma(F_1\times\cdots\times F_n))$. Thus, $\chi\big(Reg(\Gamma(F_1\times\cdots\times F_n))\big)=(|F_2|-1)\cdots(|F_n|-1)$. The proof is complete.}
\end{proof}
Now, we are in a position to prove the following result.
\begin{thm}
Let $R$ be a finite ring and $\mathfrak{m}$ be a maximal ideal of $R$ of maximum size. If $\frac{R}{\mathfrak{m}}$ 
has characteristic $2$, then
$$\chi\Big(Reg\big(\Gamma(R)\big)\Big)=\omega\Big(Reg\big(\Gamma(R)\big)\Big)=\frac{|\,{\rm Reg}(R)|}{|\frac{R}{\mathfrak{m}}|-1}.$$
\end{thm}
\begin{proof}
{First assume that $\mathfrak{m}$ is the unique maximal ideal of $R$. Since $\frac{R}{\mathfrak{m}}$  has characteristic $2$, we deduce that $2\in \mathfrak{m}$. Thus, $Reg\big(\Gamma(R)\big)$ is a disjoint union of complete graphs with $|\mathfrak{m}|$ vertices. Hence $\chi\Big(Reg\big(\Gamma(R)\big)\Big)=\omega\Big(Reg\big(\Gamma(R)\big)\Big)=|\mathfrak{m}|$. Since 
$Reg(R)=R\setminus \mathfrak{m}$, the assertion follows.  Therefore, one may assume that $|{\rm {Max}}(R)|\geq 2$. Let ${\rm {Max}}(R)=\{\mathfrak{m}_1,\ldots,\mathfrak{m}_n\}$, where $n\geq 2$ and $\mathfrak{m}_1=\mathfrak{m}$. Since  $\mathfrak{m}$ has the maximum size and $\frac{R}{\mathfrak{m}}$ has characteristic $2$, by Chinese Reminder Theorem (see \cite[Proposition 1.10]{ati}) and Lemma \ref{chromatic-no-of-reduced-rings for regs}, we obtain  
$$\chi\Big(Reg\big(\Gamma(\frac{R}{J(R)})\big)\Big)=(|\frac{R}{\mathfrak{m}_2}|-1)\cdots(|\frac{R}{\mathfrak{m}_n}|-1).$$
Moreover,  Lemma \ref{zerodivisorsofreducedsareshifted} implies that $|Reg(R)|=\big|J(R)\big|\big|Reg(\frac{R}{J(R)})\big|$. Hence, by Chinese Reminder Theorem we obtain  
$$|Reg(R)|=|J(R)|(|\frac{R}{\mathfrak{m}}|-1)(|\frac{R}{\mathfrak{m}_2}|-1)\cdots(|\frac{R}{\mathfrak{m}_n}|-1).$$
Therefore, we conclude that
$$|J(R)|\chi\Big(Reg\big(\Gamma(\frac{R}{J(R)})\big)\Big)=\frac{|\,{\rm Reg}(R)|}{|\frac{R}{\mathfrak{m}}|-1}.$$
On the other hand, by Lemma \ref{zerodivisorsofreducedsareshifted} and Lemma \ref{chromatic-no-of-reduced-rings for regs}, we deduce that
$$\omega\Big(Reg\big(\Gamma(R)\big)\Big)=|J(R)|\omega\Big(Reg\big(\Gamma(\frac{R}{J(R)})\big)\Big)=\frac{|\,{\rm Reg}(R)|}{|\frac{R}{\mathfrak{m}}|-1}.$$
Now, Lemma \ref{coloring-reduction-to-reduced-rings for regs} completes the proof.}
\end{proof}

\noindent{\bf Acknowledgements.} The authors are indebted to the School of Mathematics, Institute for
Research in Fundamental Sciences, (IPM), for support. The research
of the second author was in part supported by a grant from IPM (No.
92050212).

{}


\begin{thebibliography}{}{\small


\bibitem{akbari-heydari} S. Akbari, F. Heydari, The regular graph of a commutative ring, Period. Math. Hungar., to ppear.


\bibitem{ak} S. Akbari, D. Kiani, F. Mohammadi, S. Moradi, The total graph and regular graph of a commutative ring, J. Pure Appl. Algebra 213 (12) (2009) 2224--2228.


\bibitem{totalgraph} D.F. Anderson, A. Badawi, The total graph of a commutative ring, J. Algebra 320 (7) (2008) 2706--2719.

\bibitem{ati} M.F. Atiyah, I.G. Macdonald, Introduction to Commutative Algebra, Addison-Wesley Publishing Company, 1969.

\bibitem{bondy} J.A. Bondy, U. S. R. Murty, Graph Theory, Graduate Texts in Mathematics, 244 Springer, New York, 2008.

\bibitem{cayley-sum1} B. Green, Counting sets with small sumset, and the clique number of random Cayley graphs, Combinatorica 25 (2005) (3) 307--326.

\bibitem{cayley-sum2} D. Grynkiewicz, V.F. Lev, O. Serra, Connectivity of addition Cayley graphs, J. Combin. Theory Ser. B 99 (2009) (1) 202--217.


\bibitem{kaplansky} I. Kaplansky, Commutative Rings, rev. ed., University of Chicago Press, Chicago, 1974.


\bibitem{sharp} R.Y. Sharp, Steps in Commutative Algebra, Second edition, Cambridge University Press, 2000.


}
\end{thebibliography}
\end{document}